\documentclass[a4paper,reqno]{amsart}

  \calclayout

\usepackage{parskip}

\usepackage{hyperref}%
\hypersetup{%
  colorlinks,%
  linkcolor={red!80!black},%
  citecolor={blue!80!black},%
  urlcolor={blue!80!black}%
}

\usepackage{tikz-cd}

\usepackage{enumerate}

\usepackage{thmtools}

\declaretheorem[style=theorem,numbered=no]{theorem}
\declaretheorem[style=remark,numbered=no]{remark}
\declaretheorem[style=remark,numbered=no,name=Example]{ex}

\usepackage{mathscinet}
\newcommand{\BGP}{Bern{\v{s}}te{\u\i}n--Gel{\cprime}fand--Ponomarev}

\newcommand{\kk}{\mathbf{k}}%
\newcommand{\EE}{\mathbb{E}}%
\renewcommand{\SS}{\mathbb{S}}%
\newcommand{\op}{\mathrm{op}}

\usepackage{eucal}

\newcommand{\cat}[1]{\mathcal{#1}}%
\newcommand{\C}{\cat{C}}%
\newcommand{\D}{\cat{D}}%
\renewcommand{\S}{\mathcal{S}}%

\DeclareMathOperator{\cofib}{cofib}
\DeclareMathOperator{\Mod}{\operatorname{Mod}}

\NewDocumentCommand{\DerCat}{s m}{%
  \IfBooleanTF{#1}{\operatorname{D}}{\mathcal{D}}\!\left(#2\right)%
}

\newcommand{\noloc}{\rotatebox[origin=c]{180}{$\colon$}}

\NewDocumentCommand{\UT}{s m m m}{%
  \IfBooleanTF{#1}{%
    \begin{smallmatrix}#2&#4\\0&#3\end{smallmatrix}%
  }{%
    \left(\begin{smallmatrix}#2&#4\\0&#3\end{smallmatrix}\right)%
  }%
}

\NewDocumentCommand{\Hom}{o m m o}{%
  \operatorname{Hom}\IfValueT{#1}{_{#1}}\IfValueT{#4}{^{#4}}\left(#2,#3\right)%
}

\NewDocumentCommand{\Ext}{o m m o}{%
  \operatorname{Ext}\IfValueT{#1}{_{#1}}\IfValueT{#4}{^{#4}}\left(#2,#3\right)%
}

\NewDocumentCommand{\Map}{o m m o}{%
  \operatorname{Map}\IfValueT{#1}{_{#1}}\IfValueT{#4}{^{#4}}\left(#2,#3\right)%
}

\NewDocumentCommand{\iMap}{o m m o}{%
  \underline{\operatorname{Map}}\IfValueT{#1}{_{#1}\!}\IfValueT{#4}{^{#4}}\left(#2,#3\right)%
}

\NewDocumentCommand{\Fun}{o m m o}{%
  \operatorname{Fun}\IfValueT{#1}{_{#1}}\IfValueT{#4}{^{#4}}\left(#2,#3\right)%
}

\NewDocumentCommand{\LFun}{o m m o}{%
  \operatorname{LFun}\IfValueT{#1}{_{#1}}\IfValueT{#4}{^{#4}}\left(#2,#3\right)%
}

\NewDocumentCommand{\RFun}{o m m o}{%
  \operatorname{RFun}\IfValueT{#1}{_{#1}}\IfValueT{#4}{^{#4}}\left(#2,#3\right)%
}

\NewDocumentCommand{\PrSt}{o}{\operatorname{PrSt}^{\mathrm{L}}\IfValueT{#1}{_{#1}}}

\NewDocumentCommand{\LaxLim}{s m}{%
  \mathcal{L}\IfBooleanTF{#1}{^*}{_*}\!\left(#2\right)%
}

\title[Derived equivalences between upper-triangular ring spectra]{Derived
  equivalences of upper-triangular ring spectra\\ via lax limits}

\author[G.~Jasso]{Gustavo Jasso}%
\address{%
  Lund University, %
  Centre for Mathematical Sciences, %
  Sölvegatan 18A, %
  22100 Lund, %
  Sweden%
}%
\email{gustavo.jasso@math.lu.se}%
\urladdr{https://gustavo.jasso.info}



\begin{document}

\sloppy

\begin{abstract}
  We extend a theorem of Ladkani concerning derived
  equivalences between upper-triangular matrix rings to ring spectra. Our result
  also extends an analogous theorem of Maycock for differential graded algebras.
  We illustrate the main result with certain canonical equivalences determined
  by a smooth or proper ring spectrum.
\end{abstract}

\maketitle

The purpose of this short article is to extend the following theorem of
Ladkani~\cite{Lad11} from ordinary rings to ring spectra in the sense of stable
homotopy theory; we note that this theorem was extended to differential graded
algebras by Maycock~\cite{May11}. Recall that to rings $R$ and $S$ and an
$S$-$R$-bimodule $M$ one associates the upper-triangular matrix ring
\[
  \UT{S}{R}{M}=\left\{\left(\begin{smallmatrix}s&m\\0&r\end{smallmatrix}\right)\,\middle|\,r\in
    R,\ s\in S,\ m\in M \right\}
\]
with sum and product operations given the corresponding matrix operations. We
denote the (triangulated) derived category of right modules over a ring $R$ by
$\DerCat*{\Mod(R)}$ and recall than an object $X\in\DerCat*{\Mod(R)}$ is
\emph{compact} if the functor
\[
  \Hom[R]{X}{-}\colon\DerCat*{\Mod(R)}\longrightarrow\operatorname{Ab}
\]
preserves small coproducts.

\begin{theorem}[Ladkani]
  Let $R$ and $S$ be rings. Suppose given an $S$-$R$ bimodule $M$ such that
  $M_R$ is compact as an object of $\DerCat{\Mod(R)}$ and an $R$-module $T$ such
  that the functor
  \[
    -\otimes_E^{\mathbb{L}}T\colon\DerCat*{\Mod(E)}\stackrel{\sim}{\longrightarrow}\DerCat*{\Mod(R)}
  \]
  is an equivalence of triangulated categories, where $E=\Hom[R]{T}{T}$ is the
  ring of endomorphisms of $T$. Suppose, moreover, that $\Ext[R]{M}{T}[>0]=0$.
  Then, there is an equivalence of triangulated categories
  \[
    \DerCat*{\Mod\!\UT{S}{R}{M}}\simeq\DerCat*{\Mod\!\UT{E}{S}{\Hom[R]{M}{T}}}.
  \]
\end{theorem}

As Ladkani explains in \emph{loc.~cit.}, interesting equivalences of derived
categories are obtained from appropriate choices of $R$, $S$, $M$ and $T$. The
main focus of this article is to illustrate how formal properties of a
higher-categorical upper-triangular gluing construction yield a simple and
conceptual proof of (a vast generalisation of) the above theorem.

We use freely the theory of $\infty$-categories developed by Joyal, Lurie and
others; our main references are~\cite{Lur09,Lur17,Lur18SAG}. Here we only recall
that an $\infty$-category $\C$ is stable if it is pointed, admits finite
colimits and the suspension functor $\Sigma\colon\C\longrightarrow\C,\
X\longmapsto 0\amalg_X0,$ is an equivalence~\cite[Corollary~1.4.2.27]{Lur17}.
The homotopy category of a stable $\infty$-category is additive (in the usual
sense) and is canonically triangulated in the sense of
Verdier~\cite[Theorem~1.1.2.14]{Lur17}. Working with $\infty$-categories rather
than with triangulated categories permits us to construct the (homotopy) limit
of a diagram of exact functors between stable $\infty$-categories, a
construction that is not available in the realm of triangulated categories. We
also mention that the gluing construction that we utilise below is used by
Ladkani in~\cite{Lad11} to glue (abelian) module categories; notwithstanding,
our proof of the main theorem is different in the case of ordinary rings and of
differential graded algebras in that it does not rely on explicit computations.

Let $\kk$ be an $\EE_\infty$-ring spectrum, for example the sphere spectrum
$\SS$ or the Eilenberg--Mac Lane spectrum of an ordinary commutative
ring~\cite[Theorem~7.1.2.13]{Lur17}. The presentable stable $\infty$-category
$\DerCat{\kk}$ of $\kk$-module spectra is a (closed) symmetric monoidal
$\infty$-category~\cite[Proposition~7.1.2.7]{Lur17}. Below we work within the
symmetric monoidal $\infty$-category $\PrSt[\kk]$ of $\kk$-linear presentable
stable $\infty$-categories and $\kk$-linear colimit-preserving functors between
them~\cite[Variants~D.1.5.1 and D.2.3.3]{Lur18SAG}. Thus, an object of
$\PrSt[\kk]$ is a presentable (stable) $\infty$-category equipped with an action
of $\DerCat{\kk}$. The $\infty$-category $\PrSt[\kk]$ admits small limits and
these are preserved by the forgetful functor
$\PrSt[\kk]\to\operatorname{Pr}^{\mathrm{L}}$ to the $\infty$-category of
presentable $\infty$-categories and colimit-preserving functors between them,
see~\cite[Remark~D.1.6.4]{Lur18SAG} and~\cite[Corollary~4.2.3.3]{Lur17}. Limits
of presentable stable $\infty$-categories along colimit-preserving functors can
be computed using~\cite[Proposition 5.5.3.13 and Corollary 3.3.3.2]{Lur09} since
the limit of a diagram of stable $\infty$-categories and exact functors is
itself stable~\cite[Theorem~1.1.4.4]{Lur17}, see also~\cite[Propositions~1.1.4.1
and~4.8.2.18]{Lur17}.

Let $\C$ and $\D$ be $\kk$-linear presentable stable $\infty$-categories and
$F\colon\C\to\D$ a $\kk$-linear colimit-preserving functor. Define $\LaxLim{F}$
via the the pullback square
\[
  \begin{tikzcd}
    \LaxLim{F}\rar\dar\ar[phantom]{dr}[description,near start]{\lrcorner}&\Fun{\Delta^1}{\D}\dar{0^*}\\
    \C\rar[swap]{F}&\D
  \end{tikzcd}
\]
in the $\infty$-category $\PrSt[\kk]$; an object of the $\infty$-category
$\LaxLim{F}$ is a pair ${(c,f\colon F(c)\to d)}$ where $c\in\C$ and $f\colon
F(c)\to d$ is a morphism in $\D$. The above pullback is well defined since the
$\infty$-category $\Fun{\Delta^1}{\D}$ is
presentable~\cite[Proposition~5.5.3.6]{Lur09} and stable \cite[1.1.3.1]{Lur17}
and inherits a $\kk$-linear structure from $\D$ via the equivalence of
$\infty$-categories
\begin{align*}
  \Fun{\Delta^1}{\D}&\simeq\Fun{(\Delta^1)^\op}{\D^\op}^\op\\
                    &\simeq\LFun{\Fun{\Delta^1}{\S}}{\D^\op}^\op\\
                    &\simeq\RFun{\D^\op}{\Fun{\Delta^1}{\S}}\simeq\D\otimes\Fun{\Delta^1}{\S}.
\end{align*}
Above, $\S$ denotes the $\infty$-category of spaces, $\LFun{-}{-}$
(resp.~$\RFun{-}{-}$) denotes the $\infty$-category of functors that admit a
right adjoint (resp.~a left adjoint), and the symbol $\otimes$ denotes Lurie's
tensor product of presentable $\infty$-categories~\cite[Propositions~4.8.1.15
and 4.8.1.17]{Lur17} (see also~\cite[Theorem 5.1.5.6 and
Proposition~5.2.6.2]{Lur09}). Similarly, the restriction functor
\[
  0^*\colon\Fun{\Delta^1}{\D}\longrightarrow\Fun{\Delta^0}{\D}\simeq\D
\]
has a canonical $\kk$-linear structure. When the right adjoint $G\colon\D\to\C$
of $F$, which exists by~\cite[Corollary~5.5.2.9]{Lur09}, is also
colimit-preserving we may also form the pullback square
\[
  \begin{tikzcd}
    \LaxLim*{G}\rar\dar\ar[phantom]{dr}[description,near start]{\lrcorner}&\Fun{\Delta^1}{\C}\dar{1^*}\\
    \D\rar[swap]{G}&\C
  \end{tikzcd}
\]
in the $\infty$-category $\PrSt[\kk]$~\cite[Remark~D.1.5.3]{Lur18SAG}. There is
a canonical equivalence of $\kk$-linear presentable stable $\infty$-categories
\begin{equation}
  \label{eq:adjunction_eq}
  \LaxLim{F}\stackrel{\sim}{\longrightarrow}\LaxLim*{G},\qquad (c,f\colon
  F(c)\to d)\longmapsto(d,\overline{f}\colon c\to G(d)),
\end{equation}
stemming from the fact that both $\infty$-categories $\LaxLim{F}$ and
$\LaxLim*{G}$ are equivalent to the $\infty$-category of sections of the
biCartesian fibration over $\Delta^1$ classified by the adjunction $F\dashv G$,
see~\cite[Lemma 5.4.7.15]{Lur09}. We also remind the reader of the equivalence
of $\kk$-linear presentable stable $\infty$-categories~\cite[Lemma~1.3]{DJW21}
\begin{equation}
  \label{eq:BGP}
  \LaxLim*{F}\stackrel{\sim}{\longrightarrow}\LaxLim{F},\qquad(d,f\colon c\to F(d))\longmapsto(d,F(d)\to\cofib(f)),
\end{equation}
induced by the passage from a morphism to its cofibre, that we regard as a very
general version of the \BGP\ reflection functors~\cite{BGP73}. The gluing
operation $F\mapsto\LaxLim{F}$ is an example of a lax limit~\cite{GHN17} and is
also considered in the setting of differential graded categories, see for
example~\cite{KL15}.

For a given $\kk$-algebra spectrum $R$, that is an $\EE_1$-algebra object of the
symmetric monoidal $\infty$-category $\DerCat{\kk}$, we denote the $\kk$-linear
stable $\infty$-category of (right) $R$-module spectra by $\DerCat{R}$, see
also~\cite[Remark.~7.1.3.7]{Lur17}. The underlying stable $\infty$-category of
$\DerCat{R}$ is compactly generated by the regular representation of
$R$~\cite[Corollary~D.7.6.3]{Lur18SAG}. We identify the $\kk$-linear stable
$\infty$-category of \emph{left} $R$-module spectra with $\DerCat{R^\op}$, where
$R^\op$ denotes the opposite $\kk$-algebra spectrum
of~$R$~\cite[Remark~4.1.1.7]{Lur17}. If $M$ and $N$ are $R$-module spectra, we
denote by $\iMap[R]{M}{N}$ the $\kk$-module spectrum of morphisms ${M\to
  N}$~\cite[Example~D.7.1.2]{Lur18SAG}.

Let $R$ and $S$ be $\kk$-algebra spectra. We identify the $\infty$-category of
$S$-$R$-bimodule spectra with the $\infty$-category $\DerCat{S^\op\otimes_\kk
  R}$~\cite[Proposition~4.6.3.15]{Lur17}. The $\kk$-linear variant of the
Eilenberg--Watts Theorem~\cite[Proposition~7.1.2.4 and~p.~738]{Lur17} yields an
equivalence of $\kk$-linear presentable stable $\infty$-categories
\[
  \DerCat{S^\op\otimes_\kk
    R}\stackrel{\sim}{\longrightarrow}\LFun[\kk]{\DerCat{S}}{\DerCat{R}},\qquad
  M\longmapsto-\otimes_SM,
\]
where $\LFun[\kk]{\DerCat{S}}{\DerCat{R}}$ is the $\infty$-category of
$\kk$-linear colimit-preserving functors ${\DerCat{S}\to\DerCat{R}}$.

Given a bimodule spectrum $M\in\DerCat{S^\op\otimes_\kk R}$, we denote the right
adjoint to the tensor product functor $-\otimes_SM$ by $\iMap[R]{M}{-}$. We also
introduce the $\kk$-linear presentable stable $\infty$-category
\[
  \DerCat{\UT*{S}{R}{M}}=\LaxLim{-\otimes_SM}.
\]
The notation $\DerCat{\UT*{S}{R}{M}}$ is justified by the Recognition Theorem of
Schwede and Shipley~\cite[Corollary~D.7.6.3]{Lur18SAG} (see
also~\cite[Theorem~7.1.2.1]{Lur17}). Indeed, a standard argument using the
recollement
\[
  \begin{tikzcd}[column sep=large]
    \DerCat{R}%
    \ar[hookrightarrow]{r}[description]{i}%
    &\LaxLim{-\otimes_SM}%
    \ar[two heads]{r}[description]{p}%
    \ar[shift right=0.75em,two heads]{l}[swap]{i_L}%
    \ar[shift left=0.75em,two heads]{l}{i_R}%
    &\DerCat{S}%
    \ar[shift right=0.75em,hookrightarrow]{l}[swap]{p_L}%
    \ar[shift left=0.75em,hookrightarrow]{l}{p_R}
  \end{tikzcd}
\]
described in~\cite[Remark~1.4]{DJW21} shows that the object $X=i(R)\oplus
p_L(S)$ is a compact generator of the stable $\infty$-category
$\LaxLim{-\otimes_SM}$ whose $\kk$-algebra spectrum of endomorphisms decomposes
as the direct sum of $\kk$-module spectra
\begin{align*}
  S&\simeq\iMap{p_L(S)}{p_L(S)}&\iMap{i(R)}{p_L(S)}&\simeq M\\
  0&\simeq\iMap{p_L(S)}{i(R)}&\iMap{i(R)}{i(R)}&\simeq R,
\end{align*}
since $i_Rp_L(S)\simeq S\otimes_S M$. Upper-triangular ring spectra are
considered for example in~\cite{Sos22}.

We are ready to state and prove the main result in this article.

\begin{theorem}
  Let $R$, $S$ and $E$ be $\kk$-algebra spectra. Suppose given a bimodule
  spectrum ${M\in\DerCat{S^\op\otimes_\kk R}}$ such that the $R$-module spectrum
  $M_R=S\otimes_S M$ is compact and a bimodule spectrum
  $T\in\DerCat{E^\op\otimes_\kk R}$ such that the functor
  \[
    -\otimes_E T\colon\DerCat{E}\stackrel{\sim}{\longrightarrow}\DerCat{R}
  \]
  is an equivalence. Then, there is an equivalence of $\kk$-linear presentable
  stable $\infty$-categories
  \[
    \DerCat{\UT*{S}{R}{M}}\simeq\DerCat{\UT*{E}{S}{N}},
  \]
  where $N=\iMap[R]{M}{T}$.
\end{theorem}
\begin{proof}
  The commutative square
  \[
    \begin{tikzcd}[column sep=8em]
      \DerCat{E}\rar{\iMap[R]{M}{-\otimes_ET}}\dar[swap]{-\otimes_ET}&\DerCat{S}\dar[equals]\\
      \DerCat{R}\rar{\iMap[R]{M}{-}}&\DerCat{S}
    \end{tikzcd}
  \]
  in which the left vertical functor is an equivalence by assumption, induces an
  equivalence of $\kk$-linear presentable stable $\infty$-categories
  \begin{equation}
    \label{eq:functoriality}
    \LaxLim{\iMap[R]{M}{-}}\simeq\LaxLim{\iMap[R]{M}{-\otimes_ET}}.
  \end{equation}
  Since $\DerCat{S}$ is generated under filtered colimits by the compact
  $S$-modules~\cite[Definition~7.2.4.1 and~Proposition~7.2.4.2]{Lur17}, the
  assumption that the $R$-module spectrum $M_R=S\otimes_S M$ is compact is
  equivalent to the requirement that the (exact) functor
  \[
    \iMap[R]{M}{-}\colon\DerCat{R}\longrightarrow\DerCat{S}
  \]
  preserves small colimits~\cite[Proposition~1.1.4.1 and~1.4.4.1]{Lur17}. Hence,
  in view of the Eilenberg--Watts Theorem, the $\kk$-linear colimit-preserving
  functors
  \begin{align*}
    \iMap[R]{M}{-\otimes_ET}\colon\DerCat{E}\longrightarrow\DerCat{S}\quad\text{and}\quad-\otimes_E\iMap[R]{M}{T}\colon\DerCat{E}\longrightarrow\DerCat{S}
  \end{align*}
  are equivalent. Consequently, there is an equivalence of $\kk$-linear
  presentable stable $\infty$-categories
  \begin{equation}
    \label{eq:aux}
    \LaxLim{\iMap[R]{M}{-\otimes_ET}}\simeq\LaxLim{-\otimes_E\iMap[R]{M}{T}}.
  \end{equation}
  We conclude the proof by considering the following composite of equivalences
  of $\kk$-linear presentable stable $\infty$-categories (recall that
  $N=\iMap[R]{M}{T}$):
  \begin{align*}
    \DerCat{\UT*{S}{R}{M}}&=\LaxLim{-\otimes_SM}\\
                         &\stackrel{\eqref{eq:adjunction_eq}}{\simeq}\LaxLim*{\iMap[R]{M}{-}}\\
                         &\stackrel{\eqref{eq:BGP}}{\simeq}\LaxLim{\iMap[R]{M}{-}}\\
                         &\stackrel{\eqref{eq:functoriality}}{\simeq}\LaxLim{\iMap[R]{M}{-\otimes_ET}}\\
                         &\stackrel{\eqref{eq:aux}}{\simeq}\LaxLim{-\otimes_E\iMap[R]{M}{T}}=\DerCat{\UT*{E}{S}{N}}.\qedhere
  \end{align*}
\end{proof}

\begin{remark}
  When $\kk$ is the Eilenberg--Mac Lane spectrum of the ordinary ring of integer
  numbers, Ladkani's theorem is recovered from the previous theorem by
  considering the case where the underlying spectra of $R$, $S$, $M$ and $T$ are
  discrete, that is their stable homotopy groups vanish in non-zero degrees. The
  assumptions in Ladkani's theorem are sufficient to guarantee that the
  upper-triangular ring spectra in the statement in the previous theorem are
  both discrete. Ladkani's theorem then follows from the fact that the
  $\infty$-category of module spectra over a discrete ring spectrum $A$ is
  equivalent to the derived $\infty$-category of modules over the ordinary ring
  $\pi_0(A)$, see~\cite[Remark~7.1.1.16]{Lur17}. Maycock's extension of
  Ladkani's theorem to differential graded algebras corresponds to the case
  where $\kk$ is the Eilenberg--Mac Lane spectrum of an ordinary commutative
  ring, see~\cite[Proposition 7.1.4.6]{Lur17}.
\end{remark}

\begin{ex}
  Let $R=S=E$ be arbitrary $\kk$-algebra spectra and $M=T=R$ with its canonical
  $R$-bimodule structure. The functors ${-\otimes_RR}$ and
  $-\otimes_R\iMap[R]{R}R$ are both equivalent to the identity functor of
  $\DerCat{R}$ and the equivalence in the main theorem reduces to the
  (non-trivial) equivalence of $\kk$-linear presentable stable
  $\infty$-categories
  \[
    \DerCat{\UT*{R}{R}{R}}\simeq\Fun{\Delta^1}{\DerCat{R}}\stackrel{\sim}{\longrightarrow}\Fun{\Delta^1}{\DerCat{R}}\simeq\DerCat{\UT*{R}{R}{R}}
  \]
  given by the passage from a morphism in $\DerCat{R}$ to its cofibre.
\end{ex}

We conclude this article by describing certain canonical equivalences attached
to an algebra spectrum (or, more generally, a morphism between such) that
satisfies suitable finiteness/dualisability conditions. The bimodule spectra
that arise play a central role in the study of right/left Calabi--Yau
structures~\cite{Gin06,KS09} and their relative variants~\cite{Toe14,BD19},
see~\cite{Kel11,Yeu16,BD21,BCS23+,KW23,Wu23,Wu23a}. Given a $\kk$-algebra
spectrum $A$, we write $A^e=A\otimes_\kk A^\op$ and recall that $A$ can be
viewed either as a right or as a left $A^e$-module
spectrum~\cite[Construction~4.6.3.7 and Remark~4.6.3.8]{Lur17}. We also make
implicit use of the canonical equivalences between the $\kk$-linear
$\infty$-category of $A$-bimodule spectra and those of $A^e$-$\kk$-bimodule
spectra and of $\kk$-$A^e$-bimodule spectra,
see~\cite[Proposition~4.6.3.15]{Lur17} and the discussing succeeding it.

\begin{enumerate}[(i)]
\item\label{it:proper} Let $A$ be a proper $\kk$-algebra spectrum, that is the
  underlying $\kk$-module spectrum of $A$ is compact; equivalently, $A$ is a
  right dualisable object of the $\infty$-category of $A^e$-$\kk$-bimodule
  spectra, see~\cite[Definition~4.6.4.2]{Lur17} and~\cite[Example~D.7.4.2 and
  Remark~D.7.4.3]{Lur18SAG}. We write
  \[
    DA=\iMap[\kk]{A}{\kk}
  \]
  for the $\kk$-linear dual of $A$. Setting $R=E=\kk$, $S=A^e$, $M=A$ and
  $T=\kk$, the main theorem affords an equivalence of $\kk$-linear presentable
  stable $\infty$-categories
  \[
    \DerCat{\UT*{A^e}{\kk}{A}}\stackrel{\sim}{\longrightarrow}\DerCat{\UT*{\kk}{A^e}{DA}}
  \]
  between the derived $\infty$-category of the `one-point extension' of $A^e$ by
  the diagonal $A$-bimodule spectrum and that of the `one-point coextension' of
  $A^e$ by $DA$ (this terminology originates in representation theory of
  algebras~\cite{Rin84}).
\item\label{it:smooth} Let $A$ be a smooth $\kk$-algebra spectrum, that is
  $A\in\DerCat{A^e}$ is a compact object~\cite[Definition~11.3.2.1]{Lur18SAG};
  equivalently, $A$ is a left dualisable object of the
  $\infty$-category of $A^e$-$\kk$-bimodule spectra,
  see~\cite[Definition~4.6.4.13]{Lur17} and~\cite[Remark~11.3.2.2]{Lur18SAG}.
  The $A$-bimodule spectrum
  \[
    \Omega_A=\iMap[A^e]{A}{A^e}
  \]
  is called the inverse dualising $A$-bimodule (not to be confused with the
  based-loops functor on $\DerCat{A}$). Setting $R=E=A^e$, $S=\kk$, $M=A$ and
  $T=A^e$, the main theorem yields an equivalence of $\kk$-linear presentable
  stable $\infty$-categories
  \[
    \DerCat{\UT*{\kk}{A^e}{A}}\stackrel{\sim}{\longrightarrow}\DerCat{\UT*{A^e}{\kk}{\Omega_A}}.
  \]
\item Let $A$ be a smooth and proper $\kk$-algebra spectrum. In this case there
  are mutually-inverse equivalences of $\kk$-linear presentable stable
  $\infty$-categories
  \[
    -\otimes_A\Omega_A\colon\DerCat{A}\stackrel{\sim}{\longleftrightarrow}\DerCat{A}\noloc-\otimes_ADA,
  \]
  see~\cite[Proposition 4.6.4.20]{Lur17} where $DA$ is called the Serre
  $A$-bimodule~\cite[Definition 4.6.4.5]{Lur17} and $\Omega_A$ is called the
  dual Serre $A$-bimodule~\cite[Definition 4.6.4.16]{Lur17} (the fact that $DA$
  and $\Omega_A$ are the right and left duals of $A$ in the $\infty$-category of
  $A^e$-$\kk$-bimodule spectra in the sense of~\cite[Definition~4.6.2.3]{Lur17}
  follows from~\cite[Proposition~4.6.2.1 and Remark~4.6.2.2]{Lur17}). Setting
  $R=E=S=A$, $M=A$ and $T=DA$ or $T=\Omega_A$, the main theorem provides
  equivalences of $\kk$-linear presentable stable $\infty$-categories
  \[
    \DerCat{\UT*{A}{A}{A}}\stackrel{\sim}{\longrightarrow}\DerCat{\UT*{A}{A}{DA}}\qquad\text{and}\qquad\DerCat{\UT*{A}{A}{A}}\stackrel{\sim}{\longrightarrow}\DerCat{\UT*{A}{A}{\Omega_A}},
  \]
  where we use that $\iMap[A]{A}{DA}\simeq DA$ and $\iMap[A]{A}{\Omega_A}\simeq
  \Omega_A$ as $A$-bimodule spectra.
\item Let $f\colon B\to A$ be a morphism of $\kk$-algebra spectra that is not
  necessarily unital. By the Eilenberg--Watts Theorem, the counit of the induced
  adjunction
  \[
    -\otimes_B A\simeq f_!\colon\DerCat{B}\longleftrightarrow\DerCat{A}\noloc
    f^*
  \]
  can be interpreted as a morphism of $A$-bimodule spectra
  \[
    \varepsilon\colon A\otimes_BA\longrightarrow A.
  \]
  Suppose that $A$ is smooth and that $f^*(A)$ is compact as a $B$-module
  spectrum, so that the source and target of the morphism $\varepsilon$ are
  compact $A$-bimodule spectra and, consequently, so is its cofibre. The
  $A$-bimodule spectrum
  \[
    \Omega_{A,B}=\iMap[A^e]{\cofib(\varepsilon)}{A^e}
  \]
  is called the relative inverse dualising $A$-bimodule~\cite{Yeu16}. Setting
  $R=E=A^e$, $S=\kk$, $M=\cofib(\varepsilon)$ and $T=A^e$, the main theorem
  yields an equivalence of $\kk$-linear presentable stable $\infty$-categories
  \[
    \DerCat{\UT*{\kk}{A^e}{\cofib(\varepsilon)}}\stackrel{\sim}{\longrightarrow}\DerCat{\UT*{A^e}{\kk}{\Omega_{A,B}}}
  \]
  that specialises to the equivalence in \ref{it:smooth} when $B=0$.
\end{enumerate}

\subsection*{Acknowledgements}

The main result in this article was presented as part of a lecture series
delivered during the conference `Two Weeks of Silting' that took place in
Stuttgart, Germany, in August 2019; the author is grateful to the organisers for
the opportunity of speaking at the conference. The author thanks Peter
J{\o}rgensen for informing him of Maycock's article~\cite{May11} as well as the
anonymous referee who suggested to include further applications of the main
theorem. The author's research was supported by the Deutsche
Forschungsgemeinschaft (German Research Foundation) under Germany's Excellence
Strategy – GZ 2047/1, Projekt- ID 390685813, and partially supported by the
Swedish Research Council (Vetenskapsrådet) Research Project Grant `Higher
structures in higher-dimensional homological algebra.'

\bibliographystyle{alpha}%
\bibliography{library}

\end{document}